\documentclass[12pt]{amsart}
\usepackage{fullpage}

\usepackage{graphicx}
\usepackage{bbm}
\usepackage{amsfonts}
\usepackage{amssymb}
\usepackage{amsmath}
\usepackage{amsthm}
\usepackage{latexsym}
\usepackage[hidelinks]{hyperref}

\newtheorem{theorem}{Theorem}[section]
\newtheorem{lemma}[theorem]{Lemma}
\newtheorem{corollary}[theorem]{Corollary}
\newtheorem{proposition}[theorem]{Proposition}

\newcommand{\ignore}[1]{}

\theoremstyle{remark}
    \newtheorem{remark}[theorem]{Remark}

\newcommand{\EE}{\mathbb{E} }

\newcommand{\one}{\mathbbm{1} }
\newcommand{\PP}{\mathbb{P} }

\newcommand{\RR}{\mathbb{R} }

\newcommand{\ff}{\mathcal{F} }

\date{\today}

\title{Optimisation-based representations for branching processes}

\author{David P. Driver, Michael R. Tehranchi \\
University of Cambridge }
\address{Statistical Laboratory\\
Centre for Mathematical Sciences\\
Wilberforce Road\\
Cambridge CB3 0WB\\
UK}
\email{d.driver.maths@gmail.com \\ m.tehranchi@statslab.cam.ac.uk}

\begin{document}

\numberwithin{equation}{section}

\begin{abstract}  
 It is shown that a certain
 functional  of a branching process has representations
 in
 terms of both a maximisation problem and a minimisation problem.
 A consequence of these representation is that upper and lower bounds on the functional can be  found easily, yielding a non-asymptotic
 Trotter product formula.
As an application, the speed of the right-most particle of a branching L\'evy process
is calculated.
\end{abstract}

\maketitle

\section{Introduction}\label{se:intro}
Consider a branching process $\{ X_t^i: i \in I_t, \ t \ge 0\}$ constructed as follows.  Initially, there is one particle sitting
at a point $x_0$ in a Polish space $\mathcal X$. The position of the particle
then evolves according to the law of a given right-continuous strong Markov process $X$
started from $X_0=x_0$. At  time $T > 0$, the initial particle is killed and
replaced with $N$  particles,  where both $T$ and $N$ are random.
Each of  these new particles then move and branch as independent
copies of the initial particle, except that each new particle
now starts from  the final
position  $X_T$ of the initial particle.
 We assume that the conditional law of  the 
 first branching time  $T$ given the Markov process $X$ is
  $$
  \PP( T > t | X ) = e^{-\int_0^t \lambda(X_s) ds } \mbox{ for all } t \ge 0,
  $$
  for a given non-negative 
  measurable function $\lambda$.
We also assume that the conditional distribution of the
number of offspring $N$ given $X$ and $T$ only depends on $X_T$,
the location of the initial particle at the time of branching. 
Letting $I_t$ be the collection of particles alive at time $t$,
a construction of such a branching process $\{ X_t^i: i \in I_t,\  t \ge 0\}$ can be found in the   paper
of Ikeda--Nagasawa--Watanabe \cite{INW}.  In what follows,
we will let $X_t$ denote the position at time $t$
of the initial particle if it were allowed to continue living after the branching event. 

We will assume that the branching rate $\lambda(x)$
and the mean number of new offspring $\EE(N|X_T=x)$ 
per branching event are bounded functions of $x \in \mathcal X$. 
 This  is a sufficient
condition that the branching process does not explode in finite
time, so that $\PP(|I_t| < \infty) = 1$ for all $t \ge 0$.
See, for instance, the book of Athreya \& Ney \cite[Theorem III.2.1]{AN}

Our main result is the following:
\begin{theorem}\label{th:main-max-min}
Let $\ff$  be the filtrations generated by $X$.  Let $\mathcal Z$ be the set of 
bounded   adapted processes, let $\mathcal Z^\circ $ be the set of 
bounded    anticipative processes, and let $\mathcal M$ the
set of   non-negative martingales.  

  Given a measurable function $f : \mathcal X \to [0,1]$ and $t \ge 0$, let
  \begin{align*}
  u & = \EE  \left[ \prod_{i \in I_t} f(X_t^i) \right] \\
  M & = \max_{Z \in \mathcal Z} 
  \EE   \left[ e^{\int_0^t  Z_s ds }f( X_t)  
- \int_0^{t  }   e^{\int_0^s Z_r dr } 
h(X_s, Z_s + \lambda(X_s) ) ds \right] \\
m &= \min_{\zeta \in \mathcal M} 
  \EE   \left[\mathrm{ess \ sup}_{z \in \mathcal Z^\circ} \left\{
  e^{\int_0^t  z_s ds }f( X_t)  
- \int_0^{t  }   e^{\int_0^s z_r dr } 
\left(z_s + 
[h(X_s, z_s + \lambda(X_s)) - z_s] \frac{\zeta_{t }}{\zeta_s} \right) ds \right\}   \right] \\
   \end{align*}  
where   the function $h$ is defined by
$$
h(x,z) = \max_{0 \le \eta \le 1} \{ \eta z - \lambda(x) \EE(\eta^N | X_T = x )   \} \mbox{ for all } (x,z) \in \mathcal X \times \RR.
$$
where we set $\zeta_t/\zeta_s = 1$ on the event $\{\zeta_s = 0\}$.
Then
$$
u=M=  m .
$$
\end{theorem} 
 
 \begin{remark} We are using the convention that all real processes have
 measurable sample paths, so that the pathwise integrals appearing
 in the statement of Theorem \ref{th:main-max-min} are well-defined.
 \end{remark}

\begin{remark}\label{rmk:tildeh} The proof will show that 
it is possible to replace the function $h$ appearing in the statement of Theorem \ref{th:main-max-min} with a function $\tilde{h}$ so long as
$\tilde h(x,z) \ge h(x,z)$ for all $(x,z)$ and 
$\tilde h(x,z) = h(x,z)$ when $\lambda(x)\PP(N=1|X_T=x) \le z \le \lambda(x)\EE(N|X_T=x)$.
For instance, we may take
 $$
\tilde h(x,z) = \max_{\eta \ge 0} \{\eta z - \lambda(x)
\EE( \eta^N |X_T = x) \} 
$$
\end{remark}

The full proof  of this result appears in Section \ref{se:proof-main}.
To put the above optimisation-based representations  into context, we jump ahead a bit. The rough idea behind the equality $u=M$ appearing in 
Theorem \ref{th:main-max-min} is  that the value function
of the stochastic optimal control problem defining $M$  should  satisfy 
 the Bellman equation of the problem. However,  we 
 have chosen the data
of the control problem in such a way that the associated Bellman equation is, essentially, the S-equation (in the terminology of   Ikeda--Nagasawa--Watanabe \cite[equation (4)]{INW2}) of the branching process.  Although we have not found this done
explicitly in other papers, we acknowledge that this    may not be surprising in this respect.

In contrast, the dual minimisation problem defining $m$  is not in the form of a standard stochastic control problem, and so the usual dynamic programming
arguments do not apply. In particular, there is no conventional Bellman
equation to this minimisation problem. Our   formulation of the dual problem is inspired by
the pathwise stochastic control approach of Rogers \cite{LCGR}. The general formulation
is a bit cumbersome, involving both a  maximisation over anticipative processes and a minimisation over
martingales.
However, in the special case of dyadic branching, when the number of offspring is the constant $N  =2$, the pathwise maximisation problem can be solved 
explicitly, yielding the following corollary:
\begin{corollary}\label{th:N2} With the notation of Theorem \ref{th:main-max-min},
suppose $N=2$ almost surely.  Then
$$
u  = 1 - \max_{\zeta \in \mathcal M} \EE  \left[ \frac{ (1-f(X_t)) \zeta_t
e^{\int_0^t \lambda(X_s) ds} } {   \zeta_t  + (1-f(X_t))
 \int_0^t  \zeta_s  \lambda(X_s) e^{\int_s^t \lambda(X_r)dr} ds    } \right ]
$$
\end{corollary}
A proof of this fact will be given in section \ref{se:proof-main}.
We find it somewhat surprising (maybe even mysterious) 
that the maximisation problem appearing in Corollary \ref{th:N2} is related to a dyadic branching process. 
 
A consequence of the connection between branching processes
and the various optimisation problems  is that
lower and upper bounds of certain functionals of the branching
process can be derived immediately, simply by evaluating the objective
functions of the optimisation problems at feasible controls. In particular,
this technique can be used in principle to derive asymptotic estimates on 
the behaviour of the branching process.

As an illustrative application, we consider
the branching process where $X$ is a real-valued L\'evy process,
and where the rate of branching $\lambda$ is a positive constant
and the distribution
of the number of offspring $N$ is independent of the position of the particles.
Let  $K$ be the  cumulant generating function
of the underlying L\'evy process, defined by
 $$
 \EE_x[ e^{\theta X_t} ]= e^{ \theta x+  t K(\theta)} \mbox{ for all } t \ge 0.
 $$ 
 Suppose that $K$ is finite in a neighbourhood of $\theta  = 0$.
 Recall that by the 
L\'evy--Khintchine formula we have
$$
K(\theta) = b \theta + \frac{1}{2} \sigma^2 \theta^2
+ \int_{\RR \backslash \{0\}} [e^{\theta y} - 1 - \theta y \one_{\{|y| \le 1\}}]
\nu(dy)
$$
for some constants $b, \sigma$ and measure $\nu$,
where we are supposing that $\int (e^{  \theta   y } 
\wedge y^2 ) \nu(dy) < \infty$ for all $\theta$ in some
neighbourhood of $\theta  = 0$.

 \begin{theorem}\label{th:branch-speed}  Let  $\mu = \EE(N)-1$ be the
 mean net number of new particles created at a branching event and suppose $\mu > 0$. 
Conditional on the event $\{ I_t \ne \emptyset \mbox{ for all } t \ge 0\}$ that the branching L\'evy  process never becomes extinct, 
we have
$$
\frac{1}{t} \max_{i \in I_t} X_t^i \to  \inf_{\theta > 0} \frac{K(\theta) + \lambda \mu}{\theta} \mbox{ in probability}
$$
\end{theorem}

\begin{remark}  The condition $\mu > 0$ is necessary and
sufficient for supercriticality of the branching process, that is $\PP( I_t \ne \emptyset \mbox{ for all } t \ge 0) > 0.$ See, for instance, the book of Athreya \& Ney \cite[Theorem III.4.1]{AN}.
\end{remark}

\begin{remark}  Consider the case where the  L\'evy process
$X$ is a standard Brownian motion, so that $K(\theta) = \frac{1}{2}\theta^2$.  Then Theorem \ref{th:branch-speed} says that, conditional
on the branching process not becoming extinct, the speed
of right-most particle is 
\begin{align*}
 \inf_{\theta > 0} \left( \frac{\theta}{2} + \frac{\lambda \mu}{\theta} \right) 
 &= \sqrt{2  \lambda \mu}.
 \end{align*}
\end{remark}

\begin{remark}
Versions of Theorem \ref{th:branch-speed} are  known, see
 for instance Biggins \cite[Corollary 2]{biggins}, but our precise formulation seems new and requires fewer assumptions.  More importantly, our proof will be rather different,
  using  estimates derived from Theorem  \ref{th:main-max-min}, rather than renewal theory.
  \end{remark}

The remainder of the paper is structured as follows. Section \ref{se:proof-main}
contains the  proof of Theorem \ref{th:main-max-min}.  
The key ingredient is a more   general representation result 
given by Theorem \ref{th:main}.
Section \ref{se:bounds} gives the main take-away implications
of  Theorem \ref{th:main}: easy to apply bounds on the 
solution to certain reaction-diffusion-type equations.
Section \ref{se:speed} contains the   proof of Theorem \ref{th:branch-speed} which
finds the speed of the right-most particle of a branching L\'evy process. 

\section{Proof of Theorem \ref{th:main-max-min} }\label{se:proof-main}
In this section, we prove  Theorem  \ref{th:main-max-min}.
We first prove a more general result.  
As in the introduction, let $X$ be a right-continuous
strong Markov process valued in a Polish space $\mathcal X$.  
As in Theorem \ref{th:main-max-min}, we let $\mathcal Z$, $\mathcal Z^\circ$ and $\mathcal M$ be the set of bounded adapted processes, bounded anticipative processes and non-negative martingales, respectively.

\begin{theorem}\label{th:main}
Let $\phi:\mathcal X \times \RR \to \RR$ be measurable and such that
$\phi(x, \cdot)$ is concave and differentiable with a
derivative bounded uniformly in $x \in \mathcal X$.  Suppose the bounded function $v:\RR_+ \times \mathcal X \to \RR$ 
satisfies the integral equation
$$
v(t,x) = \EE_x\left[ v(0, X_t) + \int_0^t \phi(X_s, v(t-s, X_s) ) ds \right]
$$
for all $(t,x)$.  Then
$$
v(t,x) = \min_{Z \in \mathcal Z} \EE_x \left[
e^{\int_0^t Z_s ds} v(0,X_t) - \int_0^t e^{\int_0^s Z_r dr }
\psi( X_s, Z_s) ds \right]
$$
where 
$$
\psi(x,z) = \inf_{\eta \in \RR} \{ \eta z - \phi(x, \eta) \}
\mbox{ for all } (x,z) \in \mathcal X \times \RR.
$$
For fixed $(t,x)$, a minimiser is given by the adapted control
$$
Z_s^* = \frac{\partial \phi}{\partial v}( X_s, v(t-s, X_s) ).
$$
If $v(t,x) \ge 0$ for all $(t,x)$ and 
 $\phi(x,0 ) = 0$ for all $x  $, then
$$
v(t,x) = \max_{\zeta \in \mathcal M} \EE_x \left[
\mathrm{ess \ inf}_{z \in \mathcal Z^\circ}\left\{
e^{\int_0^t z_s ds} v(0,X_t) - \int_0^t e^{\int_0^s z_r dr }
\psi( X_s, Z_s) ds \right\} \right]
$$
For fixed $(t,x)$, a maximiser is given by the non-negative martingale
$$
\zeta^*_s = v(t-s, X_s) e^{\int_0^s \theta(X_r, v(t-r, X_r)) dr  }
$$
where $\theta(x,\eta) = \phi(x,\eta)/\eta$ for all $x \in \mathcal X, \eta>0$. For the martingale $\zeta^*$, the essential infimum is attained
for the control $z^* = Z^*$.
\end{theorem}

\begin{remark}Formally, the differential form of the
integral equation appearing Theorem \ref{th:main} is
$$
\frac{\partial v}{\partial t} = \mathcal L v + \phi(x,v)
$$
where $\mathcal L$ is the infinitesimal generator of the Markov process $X$.  The hypothesis  can be reworded to say that $v$
is a mild solution of the above differential 
  equation.
  
 We also note in passing
that when $X$ is a diffusion process in finite-dimensional Euclidean space,  
so that $\mathcal L$ is a second order differential operator,
the semi-linear partial differential equation is of the  reaction-diffusion type.
 
   Under our assumption that $\phi(x, \cdot)$ is uniformly 
Lipschitz, one can show by a standard Picard
iteration argument that given a bounded initial condition $v(0,\cdot)$
the integral equation  has a unique solution $v$ bounded
on any bounded time intervals $[0,t]$. See for instance the
 paper of Cabr\'{e} \&
   Roquejoffre \cite[Section 2.3]{cabre_influence_2013}.
   In Theorem \ref{th:main}, we take this for granted and simply
   assume that the solution $v$ exists.
\end{remark}

\begin{proof} Fix    $(t,x)$ and let 
$$
M_s = V_s  + \int_0^s  \Phi_r dr
$$
where $V_s = v(t-s, X_s)$ and $\Phi_s = \phi(X_s, V_s)$.  
Note that $(M_s)_{0 \le s \le t}$ is a martingale.

The key observation is that 
\begin{align*}
 e^{\int_0^t  Z_s ds }v(0, X_t) 
- \int_0^{t}   e^{\int_0^s Z_r dr } 
\Psi_s ds  =&
M_t +  \int_0^t ( M_t - M_s) Z_s e^{\int_0^s Z_r dr} ds
\\
& +  \int_0^t (V_s Z_s - \Phi_s - \Psi_s )e^{\int_0^s Z_r dr} ds 
\end{align*}
where $\Psi_s = \psi(X_s, Z_s)$.
Note that the two path-wise Lebesgue integrals on the right-hand side are 
well-defined, though the second one might take the value $-\infty$.  Indeed, the integrand in the first integral
is Lebesgue integrable almost surely, since by the assumed boundedness of $Z$
there is a constant $c> 0$ such that
$$
 \EE_x \left(\int_0^t |(M_t-M_s) Z_s e^{\int_0^s Z_r dr}| ds\right) 
\le c \  \EE_x (|M_t| )  < \infty
$$  
and 
$$
\EE_x \left(  \int_0^t ( M_t - M_s) Z_s e^{\int_0^s Z_r dr} ds \right)=0
$$
by Fubini's theorem  and the tower property of conditional expectation.   
The integrand in the
second integral is non-positive
by  the Fenchel--Young inequality:
 $$
 \phi(x, v) + \psi(x, z)  \le vz.
 $$
 with equality if 
 $$
 z = \frac{ \partial \phi}{\partial v}(x, v).
 $$
Hence  
 $$
\EE\left( e^{\int_0^t  Z_s ds }v(0, X_t) 
- \int_0^{t}   e^{\int_0^s Z_r dr } \Psi_s \right)
\ge \EE( M_t ) = v(t,x)
$$
with equality if $Z = Z^*$. Note $Z^*$ is bounded, and hence feasible, by the assumption of uniform
boundedness of $\partial \phi/\partial v$.  This proves that $v(t,x)$ is the value   of the  minimisation problem.

Now consider the  max-min
problem.
Fix a
non-negative martingale $\zeta$ and note that by
Fubini's theorem and iterating expectations we
have
\begin{align*}
v(t,x) & =  
\EE\left( e^{\int_0^t  Z_s^* ds }v(0, X_t) 
- \int_0^{t}   e^{\int_0^s Z_r^* dr } \psi(X_s, Z_s^*) ds \right) \\
& = 
\EE\left( e^{\int_0^t  Z_s^* ds }v(0, X_t) 
- \int_0^{t}   e^{\int_0^s Z_r^* dr } \psi(X_s, Z_s^*)  
\frac{\zeta_t }{\zeta_s} ds \right)  \\
 & \ge \EE\left( \mathrm{ess \ inf}_z \left\{ e^{\int_0^t  z_s ds }v(0, X_t) 
- \int_0^{t}   e^{\int_0^s z_r dr } \psi(X_s, z_s)   
\frac{\zeta_t }{\zeta_s} ds \right\} \right).
\end{align*}
Since $\zeta$ is arbitrary, computing the supremum of the right-hand side
yields the lower bound.  

It remains to show that there is no duality gap. 
We now assume $\phi(x,0) = 0$ for all $x$.
  Under the
uniform Lipschitz assumption, the function $\theta$ is bounded.
Let  $\Theta_s = \theta( X_s, v(t-s, X_s) )$ and 
$$
\zeta_s^* = V_s e^{\int_0^s \Theta_r dr }
$$
Note that by Fubini's theorem
$$
\zeta_s^* = M_s + \int_0^s (M_s- M_r) \Theta_r e^{\int_0^r \Theta_q dq} dr
$$
so $\zeta^*$ is a non-negative bounded martingale.
Similar to the key observation above, we have
for any anticipative process $z$ that
\begin{align*}
 e^{\int_0^t  z_s ds }v(0, X_t) 
- \int_0^{t}   e^{\int_0^s z_r dr } 
\Psi_s\frac{\zeta_t^*}{\zeta_s^*} ds  =&
\zeta_t^*  +  \int_0^t (V_s z_s - \Psi_s - \Phi_s )e^{\int_0^s z_r dr}\frac{\zeta_t^*}{\zeta_s^*} ds \\
& \ge \zeta_t^*
\end{align*}
where here $\Psi_s = \psi(X_s, z_s)$.   Note there is 
equality when $z_s = Z_s^*$ for all $s$.  This shows
\begin{align*}
\EE   \left[\min_z  \left\{
  e^{\int_0^t  z_s ds }v(0, X_t) 
- \int_0^{t  }   e^{\int_0^s z_r dr } 
\psi(X_s, z_s) \frac{\zeta_t^* }{\zeta_s^*} ds \right\}   \right]
& = \EE[ \zeta_t^* ] \\
& = v(t,x)
\end{align*}
\end{proof}

To prove Theorem \ref{th:main-max-min} we need one more
ingredient, the so-called S-equation, due to Skorokhod \cite[equation (4)]{Skorokhod}. We provide a proof for completeness.

\begin{theorem}\label{th:skorokhod}  Let $\{X_t^i: i \in I_t, t \ge 0\}$ be the branching process  described in the introduction.
  Fix a measurable $f:\mathcal X \to [0,1]$
and for all $(t,x) \in \RR_+ \times \mathcal X$, let
$$
u(t,x) =\EE_x  \left[ \prod_{i \in I_t} f(X_t^i) \right].
$$
Then 
$$
u(t,x) = 
\EE_x  \left[  f( X_t)  
 +  \int_0^{t  }   g(X_s, u(t-s, X_s) )   ds \right]
 $$
 where 
 $$
 g(x, \eta) = \lambda(x)\left( \EE[ \eta^N  | X_T = x]
 - \eta \right) \mbox{ for all } (x,\eta) \in  \mathcal X \times
[0,1].
 $$
\end{theorem}

\begin{proof}
Letting
$$
G(x,\eta) = \EE[ \eta^N |X_T = x] \mbox{ for all } (x,\eta) \in  \mathcal X \times
[0,1] 
$$
be the conditional probability generating function of 
the offspring distribution, we have
\begin{align*}
\EE_x \left[   \one_{\{t \ge T \}}   \prod_{i \in I_t} f(X_t^i)  \right]  
& = \EE_x \left[   \one_{\{t \ge T\}} u(t-T, X_T)^N  \right] \\
& = \EE_x \left[  \one_{\{t \ge T\}}  G(X_T, u(t-T, X_T))  \right] \\
& =  \EE_x \left[ \int_0^t e^{-\int_0^s \lambda(X_r)dr }\lambda(X_s)  G(X_s, u(t-s, X_s)) 
 ds \right].
\end{align*}
Hence
\begin{align*}
u(t,x) &= 
\EE_x \left[   \one_{\{t < T \}} f(X_t) + \one_{\{ t \ge T\}}   \prod_{i \in I_t} f(X_t^i)  \right] \\
& = \EE_x \left[ e^{-\int_0^t \lambda(X_s) ds} f(X_t)+
\int_0^t e^{-\int_0^s \lambda(X_r)dr }\lambda(X_s)  G(X_s, u(t-s, X_s)) 
 ds \right].
 \end{align*}
 
 Fixing $(t,x)$, the process
 $$
 M_s = e^{-\int_0^s \lambda_r dr} U_s + \int_0^s e^{-\int_0^r \lambda_v dv} \lambda_r G(X_s, U_s) dr
 $$
 is a martingale, where $\lambda_s = \lambda(X_s)$ and $U_s = u(t-s, X_s)$.
 Then by Fubini's theorem we have
 $$
 f( X_t)  
 +  \int_0^{t  }   g(X_s, U_s) )   ds
 = M_t + \int_0^t \lambda_s e^{\int_0^s \lambda_r dr} (M_t-M_s) ds.
 $$ 
 By assumption, the function $\lambda$ is bounded and hence
 the pathwise integral on the  right-hand side is
  integrable, with mean zero.  Since $\EE_x(M_t) = M_0 = u(t,x)$ we are done. 
\end{proof}

\begin{remark}  McKean \cite{mckean_application_1975}
noted that that the solution of the FKPP equation
$$
\frac{\partial u}{\partial t} = \frac{1}{2} \frac{\partial^2 u}{\partial
x^2} + u^2 - u,
$$
named after Fisher \cite{fisher_wave_1937} and Kolmogorov--Petrovskii--Piskunov   \cite{a._kolmogorov_etude_1937},
can be represented in terms of a branching Brownian motion
with unit $\lambda = 1$ branching rate and binary $N=2$ offspring
distribution. This observation is an important special case of 
Theorem \ref{th:skorokhod}.  In this context, it
is often called the McKean representation of the solution
of the FKPP equation; see the lecture notes of Berestycki \cite[Section 2.3]{B} for instance.
\end{remark}

\begin{proof}[Proof of Theorem \ref{th:main-max-min}]
Let $v(t,x) = 1- u(t,x)$ where $u(t,x)$ is defined in Theorem \ref{th:skorokhod}.  Hence, the function
$v$ satisfies the hypothesis of Theorem \ref{th:main} with 
$$
\phi(x, \eta) = - g(x, 1- \eta),
$$
where the function $g$ is defined in Theorem \ref{th:skorokhod}.
Note that $\phi(x,0) = - g(x,1) = 0$. 
The optimisation representations for $v$ yield the 
desired optimisation representations for $u$ after some manipulation.
\end{proof}

Finally, we consider the case where
the number of offspring is constant $N =2 $. 

\begin{proof}[Proof of Corollary \ref{th:N2}]
We appeal to Remark \ref{rmk:tildeh} now, and replace 
the function $h$ appearing in Theorem \ref{th:main-max-min}
with the function 
$$
h(x,z) = \max_{\eta \in \RR} \{ \eta z - \lambda(x) \eta^2 \}
=  \frac{ z^2}{4 \lambda(x) }.
$$
This yields
$$
m = 1 - \max_\zeta \EE\left[ \mathrm{ess \ sup}_z 
\left\{ e^{\int_0^t z_s ds} (1-f_t)
+ \zeta_t \int_0^t e^{\int_0^s z_r dr}
 \frac{ (z_s-\lambda_s)^2}{4 \lambda_s \zeta_s }  
 ds \right\} \right]
 $$
where we use the notation $f_t=f(X_t)$ and $\lambda_s = \lambda(X_s)$.

Letting $w_s =  \frac{1}{2} (z_s- \lambda_s)$ we see
\begin{align*}
\int_0^t \frac{e^{\int_0^s \lambda_r dr}}{\lambda_s \zeta_s}
\left(   w_s e^{  \int_0^s  w_r dr} \right)^2 ds
\ge \frac{ (e^{  \int_0^t  w_s ds} - 1)^2}{ 
\int_0^t e^{-\int_0^s \lambda_r dr}  \lambda_s \zeta_s ds}
\end{align*}
by the Cauchy--Schwarz inequality, with equality if
$ w_s e^{  \int_0^s  w_r dr} 
= - e^{-\int_0^s \lambda_r dr}\lambda_s \zeta_s 
 $.  
Also letting $W_t = e^{\int_0^t w_s ds}$ we have
by completing the square that
$$
(1-f_t) e^{\int_0^t \lambda_s ds } W_t^2
+  \frac{ \lambda_t (W_t - 1)^2}{ 
\int_0^t e^{-\int_0^s \lambda_r dr}  \lambda_s \zeta_s ds}
\ge \frac{ \zeta_t (1-f_t) e^{\int_0^t \lambda_s ds}}{\zeta_t + (1-f_t) \int_0^t e^{ \int_s^t \lambda_r dr} \lambda_s \zeta_s ds}
$$
with equality if $W_t = \frac{ \zeta_t}{\zeta_t + (1-f_t) \int_0^t e^{ \int_s^t \lambda_s}  \lambda_s \zeta_s ds}
$.

Finally, note that both equality conditions are satisfied 
for the martingale $\zeta_s^* =  v_s e^{\int_0^s \lambda_r(1-v_r)dr }$ 
and the control $w^*_s = -\lambda_s v_s$, where
$v_s = v(t-s, X_s)$ and 
$$
v(t,x) = 1 - \EE_x \left[ \prod_{i \in I_t} f(X_t^i) \right].
$$
 \end{proof}

\section{Bounding solutions}\label{se:bounds}
In this section, we explore a simple consequence  of
Theorem \ref{th:main}.  We now assume
that the  non-linearity $\phi$ appearing in the integral
equation is such that
 there is concave, differentiable and Lipschitz function $\hat \phi: \RR \to \RR$ such that 
$\phi(x, \eta) = \hat \phi(\eta)$  for all $(x, \eta) \in \mathcal X
\times \RR$.  In order
to avoid overburdening the notation, we will drop that
$\hat{}$ and simply write this function as $\phi$. 

We also introduce the following notation. We let
 $\mathbf{V}_t$ be the operator that sends the 
 bounded measurable function $v_0: \mathcal X \to \RR$ to 
 $v(t, \cdot)$, where $v: \RR_+ \times \mathcal X \to \RR$
 is the unique bounded solution to the integral equation
 $$
 v(t,x) = \EE_x\left[ v_0(X_t) + \int_0^t \phi( v(t-s, X_s)) ds \right]
 \mbox{ for all } (t,x) \in \RR_+ \times \mathcal X,
 $$
 so that $v(t,x) = \mathbf V_t(v_0)(x)$.
 We 
 let $\mathbf{P}_t$ be the transition operator
of the Markov process, such that
$$
\mathbf{P}_t (f )(x) = \EE_x [ f(X_t)] \mbox{ for all } (t,x) \in \RR_+ \times \mathcal X,
$$
for all   bounded measurable $f$. 

Finally, we let 
$$
R_t(r_0) = \mathbf{V}_t ( r_0 \one) \mbox{ for all } (t,r_0) \in \RR_+ \times \RR
$$
where $\one(x) = 1$ for all $x \in \mathcal X$.  That is to say,
if $r:\RR_+ \to \RR$ solves the ordinary differential equation
$$
\dot{r} = \phi(r), \quad r(0) = r_0
$$
then $R_t(r_0) = r(t)$.  Now let $\mathbf{R}_t$ be the operator
defined by
$$
\mathbf{R}_t(f)(x) = R_t( f(x) ).
$$

The main result of this section 
is  the following non-asymptotic
form of the Trotter product formula:

\begin{corollary}\label{th:trotter}   Fix all
 bounded measurable $f$ and integers $n \ge 1$  we have
$$
 (\mathbf{P}_{t/n} \circ\mathbf{R}_{t/n})^n (f)(x)  \le 
\mathbf{V}_t(f)(x) \le (\mathbf{R}_{t/n} \circ \mathbf{P}_{t/n})^n (f)(x)$$
for all $(t,x) \in \RR_+ \times \mathcal X.$
\end{corollary}

\begin{remark}  Our Corollary \ref{th:trotter} is very much in the spirit
of a result of Cliff, Goldstein \& Wacker \cite[Theorem 18]{CGW}, 
though our  method of proof is rather different to theirs.
\end{remark}

\begin{remark}   
Recall that a solution $v$ to the integral equation  can
be interpreted as the mild   solution to the reaction
diffusion-type equation
$$
\frac{\partial v}{\partial t} = \mathcal L v + \phi(v)
$$
where $\mathcal L$ is the generator of the Markov process $X$. 
The `diffusion' term  
corresponds to the  Markov (linear) semigroup $(\mathbf{P}_t)_{t \ge 0}$ 
 generated by $\mathcal L$, while the `reaction' term corresponds to
 the non-linear semigroup  $(\mathbf{R}_t)_{t \ge 0}$   generated by the concave (state-independent) function $\phi$.  Finally, 
$(\mathbf{V}_t)_{t \ge 0}$ is the non-linear `reaction-diffusion'
 semigroup generated by the sum $\mathcal L + \phi$.  An interesting
 reformulation of Corollary \ref{th:trotter} is
 $$
 (e^{t \mathcal L/n}e^{t \phi/n} )^n \le e^{t (   \mathcal L+\phi)}
 \le   (e^{t \phi/n} e^{t \mathcal L/n})^n 
 $$
 \end{remark}

\begin{proof} 
The key ingredient of the proof is that by Theorem \ref{th:main} 
we have
$$
R_t(r_0) = \min_{z} \left\{ e^{\int_0^t z_s ds} r_0
- \int_0^t e^{\int_0^s z_r dr} \psi(z_s) ds \right\}
$$
where
$$
\psi(z) = \inf_\eta \{ z \eta - \phi(\eta) \}.
$$
and the minimum is over deterministic bounded
measurable functions 
$z:[0,t] \to \RR$. 

We first consider the case $n=1$. For 
the upper bound, note that 
by  Theorem \ref{th:main} we have
\begin{align*}
\mathbf{V}_t(f)(x) &\le \inf_{z} \EE_x \left[
  e^{\int_0^t  z_s ds }f(X_t) 
- \int_0^{t  }   e^{\int_0^s z_r dr }    \psi( z_s)  ds \right]  \\
& = \min_{z} \left\{
  e^{\int_0^t  z_s ds }\EE_x  [f(X_t)] 
- \int_0^{t  }   e^{\int_0^s z_r dr }    \psi( z_s)  ds \right\} \\
& =  \mathbf{R}_t \circ\mathbf{P}_t (f)(x). 
\end{align*}  
 
Similarly, letting $(Z^*_s)_{0 \le s \le t}$ be the maximiser 
of the minimisation in Theorem \ref{th:main}, we have
\begin{align*}
\mathbf{V}_t(f)(x) &= \EE_x \left[
  e^{\int_0^t  Z^*_s ds }f(X_t) 
- \int_0^{t  }   e^{\int_0^s Z^*_r dr } 
  \psi( Z_s^*)  ds \right]  \\
  & \ge \EE_x \left[ \min_z \left\{
  e^{\int_0^t  z_s ds }f(X_t) 
- \int_0^{t  }   e^{\int_0^s z_r dr } 
  \psi( z_s)  ds  \right\} \right]  \\
& = \EE_x [ R_t( f(X_t) ) ] \\
& =  \mathbf{P}_t \circ\mathbf{R}_t (f)(x).
\end{align*}

Now,  note that each of the operators $\mathbf{P}_t$,
$\mathbf{R}_t$ and $\mathbf{V}_t$ are 
increasing.  In
particular, we have
\begin{align*}
 \mathbf{V}_{s+t}(f)(x) &=  \mathbf{V}_s \circ  \mathbf{V}_t (f)(x)  \\
 & \ge  \mathbf{V}_s \circ  \mathbf{P}_t \circ  \mathbf{R}_t (f)(x) \\
 & \ge \mathbf{P}_s \circ  \mathbf{R}_s \circ \mathbf{P}_t \circ  \mathbf{R}_t (f)(x).
\end{align*}
The same argument works for the upper bound. Induction completes
the proof.
\end{proof}

 \begin{remark} Alternatively, in the case where $\phi(0) = 0$, 
 we could insert the martingale $\zeta_s = 1$ 
 into the objective of the max-min problem
 in Theorem \ref{th:main} to obtain the lower bound.
 \end{remark}

\begin{remark}
From the proof of Theorem \ref{th:main-max-min}, we say that 
an interesting case is when $\phi(\eta) = \lambda( 1- \eta - G(1-\eta) )$ where $\lambda > 0$ is constant and $G$ is the probability 
generating function of a non-negative integer-valued random variable $N$. This corresponds to the case of a branching process
with a constant branching rate $\lambda$ and the distribution
of the number of particles $N$ produced at a branching event
is independent of the event's location.  In this case, we have
the formula
$$
R_t(r_0) = 1 - \EE[ (1-r_0)^{|I_t| } ]
$$
 where $I_t$ is the set of particles alive at times $t \ge 0$.
\end{remark}

\section{An application to a branching L\'evy process} \label{se:speed}
In this section, we prove Theorem \ref{th:branch-speed}.
Recall that here the branching process $\{X_t^i: i \in I_t, t \ge 0\}$
is constructed from a real-valued L\'evy process $X$
starting from $X_0 = x_0$. Given the result we are trying to prove,
there is no loss assuming $x_0=0$.  Recall also
that the branching rate is a positive constant $\lambda$ and the distribution of the number of particles $N$ produced at a branching event is independent of the position of the particles.  Recall also that the cumulant generating function $K$ of  $X$ is assumed finite in a neighbourhood 
of the origin.  In what follows, we will let
 $\hat X$ be the L\'evy process with the transition distribution
of $-X$.  Note that the function $K$ plays the role
of the Laplace exponent of $\hat X$:
$$
\EE_x [ e^{- \theta \hat{X}_t }] = e^{\theta x + t K(\theta)}
$$
for all $t \ge 0$, where here the subscript $x$ denotes
conditioning on the event $\{\hat X_0 =x\}$.

The key step of our proof of Theorem \ref{th:branch-speed} 
is the following proposition:

\begin{proposition}\label{th:front} Let $v: \RR_+ \times \RR \to [0,1]$ solve the integral equation
$$
v(t,x) = \PP_x(  \hat X_t < 0 )
+ \EE_x \int_0^t \phi( v(t-s, \hat X_s) ) ds
$$ 
where $\phi:[0,1] \to \RR$ is concave and differentiable
 with $\phi(0) = 0 = \phi(\beta)$ where $0 < \beta \le 1$, and 
 $\phi'(0) =   \gamma > 0$.
 Set
$$
q = \inf_{\theta > 0} \frac{K(\theta) + \gamma}{\theta} 
$$
Then we  have
$$
v(t, rt) \to \left\{ \begin{array}{ll} \beta & \mbox{ if } r < q \\ 0 & \mbox{ if } r > q \end{array} \right.
$$
\end{proposition} 

Before we prove Proposition \ref{th:front}, we show
how it can be used to find the asymptotic speed of the 
right-most particle:

\begin{proof}[Proof of Theorem \ref{th:branch-speed}]
Let $u(t,\cdot)$ be the distribution function of 
$M_t = \sup_{i \in I_t} X_t^i$, with $M_t= -\infty$ when $I_t$
is empty.  Note that 
by the translational invariance
of the transition distribution of a L\'evy process
\begin{align*}
u(t,x) & = \PP  ( M_t \le x) \\
&= \PP_x ( \min_{i \in I_t} \hat{X}_t^i \ge 0) \\
& = \EE_x \left[ \prod_{i \in I_t} \one_{\{ \hat{X}_t^i \ge 0 \} } \right]
\end{align*} 
According to Theorem \ref{th:skorokhod} applied to the 
branching process $\{ \hat X_t^i: i \in I_t, \ t \ge 0\}$ we have
$$
u(t,x) = \PP_x (\hat X_t \ge 0)
+ \EE_x \int_0^t g( u(t-s, \hat X_s) ) ds
$$
where
$$
g(\eta) = \lambda( \EE[\eta^N]- \eta) \mbox{ for } 0 \le \eta \le 1.
$$
Note that $g$ is convex with $g(1)= 0$ and $g(0) = \PP(N=0) \ge 0$.
By the assumption that $\EE[N] > 1$, we have $g'(1) > 0$
and hence there exists a smaller root $0 \le \alpha < 1$
such that $g(\alpha) = 0$. 

Note that $v = 1- u$ satisfies the conditions of Proposition \ref{th:front} with $\beta  = 1-\alpha$.  
Now  recall
that $\alpha = \PP( E)$ where $E = \{I_t = \emptyset \mbox{ for some } t > 0 \}$ is the event that population eventually becomes
extinct. See the book of Athreya \& Ney \cite[Theorem III.4.1]{AN}.
Hence, we have shown
$$
\PP\left( M_t \le rt \right)
\to \left\{ \begin{array}{ll} \PP(E) & \mbox{ if } r < q \\ 1 & \mbox{ if } r > q \end{array} \right.
$$
Noting that $\PP\left( \big\{M_t \le rt\big\} \cap E \right) \to \PP(E)$
since 
\begin{align*}
\PP(E) &\ge \PP\left( \big\{M_t \le rt\big\} \cap E \right) \\
 &\ge \PP\left( \big\{M_t \le rt\big\} \cap \{ I_t = \emptyset\} \right) \\
& = \PP\left(   I_t = \emptyset \right) \\
 & \to \PP( E).
 \end{align*}
 the conclusion follows since
\begin{align*}
\PP\left( M_t \le rt \ \big| \ E^c \right)
& = \frac{1}{\PP(E^c)}[ \PP\left( M_t \le rt \right) -  \PP\left( \{M_t \le rt\}
\cap E \right)] \\
&\to \left\{ \begin{array}{ll} 0 & \mbox{ if } r < q \\ 1 & \mbox{ if } r > q. \end{array} \right.
\end{align*}
This shows that for any $\varepsilon > 0$ we have
$$
\PP( | \tfrac{1}{t} M_t - q| > \varepsilon \ | \ E^c )
 \to 0
 $$
 as desired.
\end{proof} 
 
  The rest of this section contains the proof of Proposition
  \ref{th:front}.  The case where $\hat X$ is degenerate,
  in the sense that $\hat X_t = x + bt$ for a constant $b$ is
  immediate. Therefore, we will assume without loss that $\hat X$ is non-degenerate, so that $K''(0) = \mathrm{Var}(\hat X_1) > 0$.

Of the two bounds, the upper bound is easier to obtain. 
 Using the $n=1$ case of Theorem~\ref{th:trotter}, 
 we have
$$
v(t,x) \le R_t( \PP_x (X_t <  0) ).
$$
By the concavity of $\phi$ we have
$$
\phi(v) \le \gamma v
$$
and hence by Gr\"onwall's inequality
$$
R_t(r_0) \le r_0 e^{\gamma t}.
$$
Now by Markov's inequality we have
$$
\PP_x( \hat X_t < 0 )   \le e^{ - x \theta + t K(\theta) }
$$
for any $\theta > 0$ and $t \ge 0$.  Putting this together, we have shown
$$
v(t, rt)  \le e^{ t(  K(\theta) +\gamma  - r \theta) }.
$$
If $r >  \frac{1}{\theta}( \Lambda(\theta) + \gamma) $ then
the right-hand side vanishes as $t \to \infty$, as claimed.

For the lower bound, we will introduce some more notation.
Let 
$$
F_t(y) = \PP_0( \hat X_t \le y)
$$ 
be the conditional distribution function of the random variable 
$\hat X_t$ given $\hat X_0=0$.  Note that by spacial homogeneity of the L\'evy process, we have
$$
\PP_x(\hat X_t \le y) = F_t(y-x).
$$
Let $F_t^{-1}$ be the quantile function, defined as
$$
F_t^{-1}(p) = \inf\{ x: F_t(x) \ge p \},
$$
 so that $F_t(x) \ge p \Leftrightarrow x \ge F_t^{-1}(p)$.

The key estimates are the following:
\begin{lemma}\label{th:lower} For all $0 <b < \beta $, $n \ge 1$, $t > 0$ and $x \in \RR$ we
have 
$$
v(t,x) \ge b F_\delta\left( - x - (n-1) F_{\delta}^{-1}\left(
\frac{Q_{\delta}^{-1}(b)}{b} \right) \right)
$$
where $  \delta = t/n$.
\end{lemma}

\begin{remark} It is interesting to note that Lemma \ref{th:lower} actually
holds with no assumption on law of the L\'evy process.  In particular,
it holds for processes, such as stable processes, for which the 
Laplace exponent $K(\theta)$ is infinite 
for all $\theta \ne 0$.
\end{remark}

\begin{proof}[Proof of Lemma \ref{th:lower}]
We fix $\delta$ and use induction on $n$. We first consider the $n=1$ case.

Since the points $0$ and $\beta \le 1$ are fixed points of $\phi$, we have
$R_\delta(0) = 0$ and  $R_{\delta}(1) \ge \beta$. In particular, we have
\begin{align*}
v(\delta,x) &\ge \mathbf{P}_\delta \circ 
\mathbf{R}_\delta  \one_{(-\infty, 0]}(x)\\
&  \ge  \beta \PP_x( \hat X_\delta \le 0) \\
& = \beta F_\delta(-x)
\end{align*}

To do the inductive step, we will make use of the following
observation:  for any $0< b < \beta$ 
and $k \in \RR$ we have
$$
R_{\delta} [ b F_\delta (k) ] \ge b \one_{\{ F_{\delta}(k) \ge R_{\delta}^{-1}(b)/b\}}
$$
since  $R_{\delta}$ is increasing on $[0, \beta]$.
Now suppose the claim is true for $n=m$, we have
\begin{align*}
v((m+1)\delta,x) &\ge 
\mathbf{P}_{\delta} \circ 
\mathbf{R}_{\delta} \left[ b F_\delta\left( - \cdot - (m-1) F_{\delta}^{-1}\left(
\frac{R_{\delta}^{-1}(b)}{b} \right) \right) \right](x) \\
   &\ge 
b \ \PP_x \left[  F_\delta\left( - \hat X_{\delta}- (m-1) F_{\delta}^{-1}\left(
\frac{R_{\delta}^{-1}(b)}{b} \right) \right) \ge \frac{R_{\delta}^{-1}(b)}{b}\right] \\
&= b \ F_\delta\left( -x - m F_{\delta}^{-1}\left(
\frac{R_{\delta}^{-1}(b)}{b} \right) \right).
\end{align*}

\end{proof}

\begin{lemma}\label{th:first} For all  $0< c < \gamma = \phi'(0)$ and 
all $0 < b< \beta$, where $\beta$ is the larger root of $\phi$, 
 there exists $\delta^* > 0$
such that $R_\delta^{-1}(b) \le b e^{-c\delta}$  for all $\delta \ge \delta^*$.
\end{lemma}

\begin{proof}  Fix a $q^* \in (0, \beta)$, for instance $q^* = \beta/2$ and let
$$
H(q) = \int_{q^*}^q \frac{ ds}{\phi(s)}.
$$
Note that the differential equation defining $R$ can be solved as
$$
R_{\delta}(r_0) = H^{-1}( H(r_0) + \delta)
$$
for $0 < r_0 < \beta$, and hence
$$
R_{\delta}^{-1}(r_0) = H^{-1}( H(r_0) - \delta).
$$
In this notation, we must prove that 
$$
H(b) - \delta \le H(  b e^{-c\delta})
$$
 or equivalently
$$
\frac{1}{\delta} \int_0^\delta \frac{ b c e^{-cx} dx}{\phi( be^{-cx}) } \le 1
$$
for $\delta$ large enough.  To do this, note that the limit
of the left-hand side as $\delta \to \infty$ is $c/\gamma < 1$ by
l'H\^opital's rule. 
\end{proof}

\begin{lemma}\label{th:second} For all $r < q$
there exists a  $c  < \gamma$ and a $\delta^* > 0$ such that
$F_{\delta}^{-1}( e^{-c \delta}  ) \le - r \delta$ for all  $\delta \ge \delta^*$.
\end{lemma}

\begin{proof} 
Note that $q > -\EE_0(\hat X_1)$  with strict inequality since 
since $K(\theta) > - \theta \EE_0(\hat X_1)$ by Jensen's inequality. 
Hence we need only consider $r$ such that
$$
-\EE_0(\hat X_1) < r < q.
$$
 In particular, we may invoke 
 Cram\'er large deviation principle to conclude that, 
 $$
\log F_{\delta}(-r \delta) = - \hat K(r) \delta (1 + o(1) )
$$
as $\delta \to \infty$, where the large deviation rate
function $\hat K$ is the Legendre transform of
$K$, defined by
$$
\hat K ( \eta ) = \sup_{\theta} \left[ \eta \theta - K (\theta)\right].
$$
Hence, it is enough to show that 
$$
\hat K(r) < \gamma.
$$

Now, since $r > K'(0) = - \EE_0(\hat X_1)$, there exists an
$\varepsilon > 0$ such that $r >  K'(\varepsilon)$, since
$K'$ is continuous and increasing in a neighbourhood of $\theta = 0$.  By the 
convexity of $K$ we have the inequality
$$
r \theta - K(\theta) \le r \varepsilon - K(\varepsilon)
$$
for $\theta < \varepsilon$ and hence
\begin{align*}
\hat K(r) & = \sup_{ \theta \ge \varepsilon}   \left[ r \theta - K (\theta) \right] \\
& \le -\varepsilon(q-r) +  \sup_{ \theta \ge \varepsilon}   \left[ q \theta - K (\theta) \right].
\end{align*}
The conclusion follows since $q \theta - K (\theta) \le \gamma$ for
all $\theta > 0$ by the definition of $q$.
\end{proof}

\begin{proof}[Proof of Proposition \ref{th:front}] 
 Fix $0 < b < \beta$ and $r < q$. Pick $\bar r$
 such that $r < \bar{r} < q$. 
 By Lemma \ref{th:second}
 there exists a $c$ and $\delta_1^*$ such that 
$F_{\delta}^{-1}( e^{-c \delta}) \le - \bar{r} \delta$ for all $\delta \ge \delta^*_1$.  
By Lemma \ref{th:first} there exists $\delta_2^*$
such that $R_\delta^{-1}(b) \le b e^{-c\delta}$  
for all $\delta \ge \delta_2^*$. 

Let $m= 1 + \EE_0(\hat X_1)$. 
  By the weak law of large numbers
$$
F_\delta( m \delta) = \PP_0( \hat X_{\delta}/\delta \le m ) \to 1.
$$
So given $\varepsilon > 0$, there exists $\delta_3^*$ such that 
$F_\delta( m \delta) \ge 1- \varepsilon$ for $\delta \ge \delta_3^*$.

Let $n \ge \frac{\bar{r} + m }{\bar{r} - r}$  and 
 $t \ge n \max_i \{ \delta_i^* \}$.   Finally, let $\delta = t/n$, so $\delta \ge \max_i \{ \delta_i^* \}$
and hence
\begin{align*}
v(t, rt) &= v(n \delta,  r n \delta) \\
& \ge b F_{\delta}(- r n \delta - (n-1) F_{\delta}^{-1}\left( R_\delta^{-1}(b)/b \right)  ) \\
& \ge b F_{\delta}(- r n \delta - (n-1) F_{\delta}^{-1}(e^{-c\delta}) ) \\
& \ge b F_{\delta}( - r n \delta + (n-1) \bar{r} \delta) \\
& = b F_{\delta}( [n(\bar{r}-r) - \bar{r}] \delta ) \\
& \ge b F_\delta( m \delta)\\
& \ge b (1- \varepsilon).
\end{align*} 
Since $b < \beta$ and $\varepsilon > 0$ are arbitrary, the conclusion follows.
\end{proof}

 \begin{remark}
It is possible to express  the speed  $q$ of the travelling 
wave front in several ways. The above
 proof show that $q$ can be
rewritten as
$$
q = \sup\{ r: \hat K(r) < \gamma \},
$$
where  $\hat K$ is  the Legendre  transform of $\Lambda$.
This formulation for the speed of the right-most particle 
appears in the paper of Biggins \cite{biggins} or, more recently, in
the paper of Groisman \&  Jonckheere \cite{GJ}.

Following an idea in the paper of  Hiriart-Urruty \&  Mart\'inez-Legaz  \cite{HM}, an inverse to the function $\hat K$ can
be calculated as follow. First, define a new function $K^\circ$ by the formula
$$
K^\circ( \theta) = \left\{ \begin{array}{ll} +\infty & \mbox{ if } \theta \ge 0 \\
	  -\theta K(-1/\theta) & \mbox{ if } \theta < 0.
	\end{array} \right.
$$
Note that the function $K^\circ$ is convex, and indeed, it is related to the perspective
function of the Laplace exponent $K$.   Define its Legendre transform in the usual fashion
$$
\hat K^\circ( \eta ) = \sup_{\theta}\left[ \eta \theta - K^{\circ}(\theta)\right].
$$
Then it can be shown that an inverse function to $\hat K$ is the function $-\hat K^{\circ}$. In particular, the speed $q$ can be rewritten as
$$
q = - \hat K^\circ( \gamma ).
$$  
Simplifying the above formula recovers the formula in Proposition \ref{th:front}.
\end{remark}

\begin{remark}
Consider the case of dyadic branching Brownian motion, where
$N=2$ and $K(\theta) = \frac{1}{2}\theta^2$.  
Letting $m(t)$ be the median, defined by 
$$
\PP( \max_{i \in I_t} X_t^i \le m(t)) = 1/2
$$
we have
\begin{equation}\label{eq:median}
- \sqrt{t}  \Phi^{-1}\left(\frac{1 }{e^{t}+1} \right) \ge 
m(t) 
  \ge - \frac{\sqrt{t}}{\sqrt{n }}  \Phi^{-1}\left( \frac{1}{2b}\right) - \frac{(n-1)\sqrt{t}}{\sqrt{n }}  
  \Phi^{-1}\left( \frac{1}{e^{t/n}(1-b)+b} \right)  
 \end{equation}
for all $1/2 < b < 1, n \ge 1$, where 
$$
\Phi(z) = \int_{-\infty}^z \frac{e^{-s^2/2}}{\sqrt{2\pi}} ds
$$
is the standard normal distribution function. Indeed, the upper bound 
follows from the upper bound $1/2  \le v(t,m) =  R_t [ \PP(X_t \le -m ) ]$
and the calculation $R_t(r_0) = \frac{r_0}{r_0 + e^{-t}(1-r_0)}$ in the case when
 $\phi(v) = v(1-v)$. 
The lower bound is implied by Lemma \ref{th:lower}.

  Using $\Phi^{-1}(\varepsilon) = - \sqrt{2 \log(1/\varepsilon)} 
(1 + o(1) )$ as $\varepsilon \downarrow 0$ yields 
$$
m(t) = \sqrt{2}t + o(t).
$$
On the other hand, a famous result of Bramson \cite{bramson_maximal_1978} says
$$
m(t) = \sqrt{2}t - \frac{3}{2 \sqrt{2}} \log t + O(1).
$$
Since $\Phi^{-1}(\varepsilon) = - \sqrt{2 \log(1/\varepsilon)} 
+ O\left( \frac{\log \log( 1/\varepsilon) }{\sqrt{  \log(1/\varepsilon)}} \right)
 $
 the upper bound in equation \eqref{eq:median} actually
 recovers the correct order of magnitude of the second term of the expansion.  
It would be interesting to see if, by optimising over the free
parameters $b$ and $n$, it is possible to recover
the $\log t$ term in the lower bound as well.
\end{remark}

 \section{Acknowledgements}
David Driver was supported, in part, by the UK Engineering and Physical Sciences Research Council (EPSRC) grant EP/H023348/1 for the University of Cambridge Centre for Doctoral Training, the Cambridge Centre for Analysis.
Mike Tehranchi thanks the Cambridge Endowment for Research in Finance for 
their support.  Finally, both of us thank Nathanael Berestycki for his comments and 
suggestions.

\end{document}